\documentclass[a4paper]{article}
\usepackage{hyperref}
\usepackage{amsmath}
\usepackage[dvips]{graphicx}
\usepackage{amsthm}
\usepackage{epsfig}
\usepackage{amssymb}
\usepackage{dsfont}
\usepackage{float}

\newtheorem{theorem}{Theorem}[section]

\newtheorem{lem}{Lemma}[section]

\newtheorem{de}{Definition}[section]

\begin{document}

\begin{center}
{\Large  On critical exponents of a $k$-Hessian equation in the whole space}
\end{center}
\vskip 5mm

\begin{center}
{\sc Yun Wang and Yutian Lei} \\
\vskip2mm
Institute of Mathematics,
School of Mathematical Sciences,\\
Nanjing Normal University,
Nanjing, 210023, China
\vskip 5mm
\end{center}

\vskip 5mm {\leftskip5mm\rightskip5mm \normalsize
\noindent{\bf{Abstract}} In this paper, we study negative
classical solutions and stable solutions of the following
$k$-Hessian equation
$$
F_k(D^2V)=(-V)^p \quad in~R^n
$$
with radial structure, where $n \geq 3$, $1<k<n/2$ and $p>1$. This
equation is related to the extremal functions of the Hessian
Sobolev inequality on the whole space. Several critical exponents
including the Serrin type, the Sobolev type, and the
Joseph-Lundgren type, play key roles in studying existence and
decay rates. We believe that these critical exponents still come
into play to research $k$-Hessian equations without radial
structure.

 }

\newtheorem{proposition}[theorem]{Proposition}

\renewcommand{\theequation}{\thesection.\arabic{equation}}
\catcode`@=11
\@addtoreset{equation}{section}
\catcode`@=12

\section{Introduction}

In 1990, Tso \cite{Tso1} studied the relation between the value of
exponent $p$ and the existence results for the $k$-Hessian
equation $F_k(D^2V)=(-V)^p$ in bounded domains. The critical
exponent $p=\frac{(n+2)k}{n-2k}$ plays a key role. Those results
are associated with the extremal functions of the Hessian Sobolev inequality for all
$k$-admissible functions which was introduced by Wang in
\cite{Wang}. Such an inequality with the critical exponent still
holds in the whole space $R^n$, and the extremal functions are
radially symmetric (cf. \cite{PDE}, \cite{TW}).

Consider the Euler-Lagrange equation
\begin{equation} \label{hesse}
F_k(D^2V)=(-V)^p, \quad V<0~in~R^n,
\end{equation}
with a general exponent $p>1$, where $n \geq 3$, $1<k<n/2$. Here
$F_k[D^2V]=S_k(\lambda(D^2V))$,
$\lambda(D^2V)=(\lambda_1,\lambda_2,\cdots,\lambda_n)$ with
$\lambda_i$ being eigenvalues of the Hessian matrix $(D^2V)$, and
$S_k(\cdot)$ is the $k$-th symmetric function:
$$
S_k(\lambda)=\sum_{1\leq i_1<\cdots<i_k \leq n}
\lambda_{i_1}\lambda_{i_2}\cdots\lambda_{i_k}.
$$
According to the conclusions in \cite{CNS}, $V<0$ ensures that the
main part of (\ref{hesse}) is elliptic. Namely, we always consider
the $k$-admissible solutions in the cone
$$
\Phi^k:=\{u \in C^2(R^n);F_s(D^2V) \geq 0, s=1,2,\cdots,k\}.
$$
Such an equation does not only come into play to study the
extremal functions of the Hessian Sobolev inequality, but also is
helpful to investigate the global existence and blow-up in finite
time span for the fully nonlinear parabolic equations (such as the
equations studied in \cite{IL}, \cite{Ren} and \cite{WL}).

A special case is $F_1[D^2V]=\Delta V$, and (\ref{hesse})
becomes the Lane-Emden equation
\begin{equation} \label{LE}
-\Delta u=u^p,  \quad u>0~in~R^n.
\end{equation}
The existence results of the solutions of this equation have
provided an important ingredient in the study of conformal
geometry, such as the extremal functions of the Sobolev
inequalities and the prescribing scalar curvature problem. It was
studied rather extensively. According to Theorem 3.41 in
\cite{Ni}, (\ref{LE}) has no positive solution even on exterior
domains when $p$ is not larger than the Serrin exponent (i.e. $p
\in (1,\frac{n}{n-2})$). The Liouville theorem in \cite{GS} shows
that (\ref{LE}) has no positive classical solution in the
subcritical case (i.e. $p \in [1,\frac{n+2}{n-2})$). In the
critical case (i.e. $p=\frac{n+2}{n-2}$), the positive classical
solutions of (\ref{LE}) must be of the form
\begin{equation} \label{wlty}
u(x) =c(\frac{t}{t^2+|x-x^*|^2})^{\frac{n-2}{2}}
\end{equation}
with constants $c,t>0$, and $x^* \in R^n$ (cf. \cite{CL}). In supercritical case
(i.e. $p>\frac{n+2}{n-2}$), existence and asymptotic behavior of positive solutions
are much complicated and not completely understood. In fact, we can find cylindrical shaped
solutions which does not decay along some direction.
In addition, there are radial solutions with the slow decay rates solving (\ref{LE}) (cf.
\cite{GS}, \cite{Joseph}, \cite{WangXF} and many others).
Furthermore, those radial solutions are of the form
$$
u(x)=\mu^{\frac{2}{p-1}}U(\mu |x|), \quad x \in R^n,
$$
where $\mu=u^{\frac{p-1}{2}}(0)$, and $U(r)$ is the unique
solution of
$$
 \left \{
   \begin{array}{l}
      -(U''+\frac{n-1}{r}U')=U^p,      \quad U(r)>0, ~r>0\\
      U'(0)=0, \quad U(0)=1.
   \end{array}
   \right.
$$

For the study of `stable' positive solutions of (\ref{LE}), the
Joseph-Lundgren exponent
$$
p_{jl}(n):=1+\frac{4}{n-4-2\sqrt{n-1}}
$$
plays an important role (cf. \cite{GNW}). Such an exponent is also
essential to describe how the radial solutions intersect with the
singular radial solution and with themselves (cf. \cite{Joseph}).
In addition, this Joseph-Lundgren exponent can be
used to study the Morse index for the sign-changed solutions of
the Lane-Emden equation (cf. \cite{AF}) and other nonlinear
elliptic equations with supercritical exponents (cf. \cite{AL} \cite{ED} and
\cite{GW}).

In this paper, our purpose is to study the relation between the
critical exponents and existence of kinds of solutions of $k$-Hessian
equation (\ref{hesse}). As the beginning of the study, we are
concerned about the increasing negative solution
of (\ref{hesse}) with radial structure as in \cite{PDE} and
\cite{YM}. Thus, (\ref{hesse}) is reduced to the
following equation
\begin{equation}\label{1.1}
-\frac{1}{k}C_{n-1}^{k-1}(r^{n-k}|u'|^{k-1}u')'=r^{n-1}u^{p},
\quad u(r)>0\quad as\quad r>0.
\end{equation}
Here $u(r)=u(|x|)=-V(x)$, $n \geq 3$, $1<k<n/2$  and $p>1$.
In fact, in the critical case (i.e. $p=\frac{(n+2)k}{n-2k}$),
the extremal functions of the Hessian Sobolev inequality
are radially symmetric (cf. \cite{PDE}, \cite{TW} and \cite{Wang}).
In the noncritical case,
it is clearer and more concise to study the critical exponents of
the radial solutions. We believe that the ideas are helpful to investigate the corresponding
problems of the solutions with general form, and those critical exponents
still come into play in the study of $k$-Hessian equations without radial structure.

\subsection{Regular solutions}

Clearly, (\ref{1.1}) has a singular solution
\begin{equation}\label{US}
u_{s}(r)=Ar^{-\frac{2k}{p-k}},\quad with \quad
A:=(\frac{1}{k}C^{k-1}_{n-1})^{\frac{1}{p-k}}(\frac{2k}{p-k})^{\frac{k}{p-k}}
(n-\frac{2pk}{p-k})^{\frac{1}{p-k}}.
\end{equation}
If write $V(x)=-u_s(|x|)$, then $V(x)$ only belong to
$C^2(R^n\setminus \{0\})$ (even it does not belong to $L_{loc}^\infty(R^n)$).

We are mainly concerned with the $k$-admissible solutions of
(\ref{hesse}). Consider the following boundary values problem
\begin{equation}\label{1.6}
 \left \{
   \begin{array}{l}
      -\frac{1}{k}C_{n-1}^{k-1}(r^{n-k}|u'|^{k-1}u')'=
      r^{n-1}u^{p},\quad u(r)>0, ~r>0\\
      u'(0)=0, \quad u(0)=\rho(:=\mu^{\frac{2k}{p-k}}) >0.
   \end{array}
   \right.
\end{equation}
\begin{de}
If a solution $u(r)$ of (\ref{1.6}) satisfies $u(|x|) \in C^2(R^n)$,
then $u(r)$ is called a {\it regular} solution.
\end{de}

Recall two critical exponents:
Serrin exponent $p_{se}:=\frac{nk}{n-2k}$, and Sobolev exponent $p_{so}:
=\frac{(n+2)k}{n-2k}$.

When $p$ is not larger than the Serrin exponent,
(\ref{hesse}) has no negative $k$-admissible solution (cf.
\cite{Lei2}, \cite{Ou} and \cite{PV2}). Thus, we always assume in
this paper that $p$ is larger than the Serrin exponent
\begin{equation} \label{serrin}
p>p_{se}.
\end{equation}

In the critical case (i.e. $p=p_{so}$), all the
regular solutions of (\ref{1.6}) can be written as the explicit form (cf.
Remark 1.4 in \cite{YM})
\begin{equation} \label{expl}
u_\rho(r)=(\frac{1}{k}C_{n-1}^{k-1})^{\frac{1}{p-k}}
\rho(1+\frac{k}{n^{1/k}(n-2k)}(\rho^{\frac{k+1}{n-2k}}r)^2)^{-\frac{n-2k}{2k}}.
\end{equation}
Therefore, we will be concerned with the noncritical cases.

\begin{theorem} \label{th1.1}
When $p<p_{so}$,  (\ref{1.6}) has no regular solution.
\end{theorem}

\paragraph{Remark 1.1.}
By a direct calculation, when $p_{se}<p<p_{so}$, besides $u_s$ given by
(\ref{US}), (\ref{1.1}) has other
singular solutions $U_s(r)$ satisfying $U_s(r)/u_s(r) \to 1$ as $r \to 0$
and $U_s(r)r^{\frac{n-2k}{k}} \to \lambda>0$ as $r \to \infty$.
When $k=1$, this result can be found in \cite{GS}, \cite{Joseph}, \cite{WangXF}.

\begin{theorem} \label{th1.3}
When $p > p_{so}$, all the positive regular solution $u_\mu$ of (\ref{1.6})
satisfies $u_\mu(r) \simeq r^{-\frac{2k}{p-k}}$ for large $r$. Furthermore,
they are the form of
\begin{equation}\label{xiu0}
u_\mu(r)=\mu^{\frac{2k}{p-k}}u_1(\mu r), \quad r \geq 0,
\end{equation}
where $u_1(r)$ is the solution of
\begin{equation}\label{jia6}
 \left \{
   \begin{array}{l}
      -\frac{1}{k}C_{n-1}^{k-1}(r^{n-k}|u'|^{k-1}u')'=
      r^{n-1}u^{p},\quad u(r)>0, ~r>0\\
      u'(0)=0, \quad u(0)=1.
   \end{array}
   \right.
\end{equation}
\end{theorem}
Here, $u(r) \simeq r^{-\theta}$ means that there exists $C>1$ such
that $\frac{1}{C}\leq u(r)r^{\theta}\leq C$ for large $r$.

\paragraph{Remark 1.2.} Problem (\ref{jia6}) has a entire solution when $p>p_{so}$.
In fact, by a standard argument of contraction, (\ref{jia6}) has a
unique local positive solution $u$ (cf. Proposition 2.1 in \cite{YM}).
There holds $u'<0$ as long as $u>0$ (see the proof of Lemma \ref{lem2.1}).
Extend this local solution rightwards. Then $u>0$ for all $r>0$.
Otherwise, it contradicts with the Liouville theorem in \cite{Tso1}.

\subsection{Stable solutions}

\begin{de}
We say that a positive solution $u\in C^{1}(0,\infty)$ of
(\ref{1.1}) is stable if
\begin{equation}\label{weak1}
\int_0^\infty
[\frac{1}{k}C_{n-1}^{k-1}r^{n-k}|u'|^{k-1}u'\varphi'-
      r^{n-1}u^{p}\varphi]dr=0;
\end{equation}
\begin{equation}\label{weak2}
Q_{u}(\varphi):=  C_{n-1}^{k-1}\int^{\infty}_{0}
r^{n-k}|u'|^{k-1}(\varphi')^2dr-p\int^{\infty}_{0}
r^{n-1}u^{p-1}\varphi^2 dr \geq 0
\end{equation}
for all $\varphi \in W_*$, where
$W_*=\{\varphi(r);\varphi(r)=\phi(x)\in C^{\infty}_
{c}(R^{n}),r=|x|\}$.

Similarly, a positive solution $u\in C^{1}(0,\infty)$ of
(\ref{1.1}) is stable on a set $(R,\infty)$ for some $R>0$, if
(\ref{weak1}) holds for all $\varphi \in W_*$,
and (\ref{weak2}) holds for all $\varphi \in C_c^\infty(R,\infty)$.
\end{de}

Indeed, the fact that the first order Fr$\acute{e}$chet derivative of the
functional $J(u)$ is equal to zero and the second order Fr$\acute{e}$chet derivative
is nonnegative can lead to this definition, where
$$
J(u)=\frac{C_{n-1}^{k-1}}{k(k+1)}\int_0^\infty |u'|^{k+1}r^{n-k}dr
-\frac{1}{p+1}\int_0^\infty u^{p+1}r^{n-1}dr.
$$
In addition, $Q_u(\varphi) \geq 0$ can also be obtained by linearizing (\ref{1.1}).

It is not difficult to verify that the regular solutions $u_\rho$
given by (\ref{expl}) and $u_\mu$ given by (\ref{xiu0}) satisfy (\ref{weak1}).
For the singular solution $u_s$ expressed by (\ref{US}), $p>p_{se}$ implies that $0$ is not
the singular point in integral terms of (\ref{weak1}) (see the proof of Theorem \ref{th1.5}).
Therefore, $u_s$ also satisfies (\ref{weak1}).

Recall other two critical exponents: the Joseph-Lundgren exponent
$$
p_{jl}=\left\{
 \begin{array}{lll}
 &\displaystyle  \infty, & if N\leq2k+8,\\[3mm]
 &\displaystyle\frac{k[n^2-2(k+3)n+4k]
+4k\sqrt{2(k+1)n-4k}}{(n-2k)(n-2k-8)}, & if N > 2k+8;
  \end{array}
 \right.
$$
and
$$
p^*=k\frac{n+2k}{n-2k}.
$$
Clearly, $p_{se}<p_{so}<p_{jl}$.
In addition, $p_{so}<p^*$ by virtue of $1<k<n/2$. In view of
$2k(k^2+6k+1)/(k-1)^2>2k+8$, we can
deduce the relation between $p^*$ and $p_{jl}$ as follows
$$\begin{array}{ll}
&p^* \geq p_{jl}, \quad if ~n \geq 2k(k^2+6k+1)/(k-1)^2;\\
&p^* < p_{jl}, \quad if ~n < 2k(k^2+6k+1)/(k-1)^2.
\end{array}
$$

Under the scaling transformation, $p=p_{so}$ ensures that equation (\ref{hesse})
and energy $\|\cdot\|_{p+1}$ are invariant (cf \cite{Lei2}), and
$p=p^*$ ensures that equation (\ref{hesse}) and energy $\|\cdot\|_{p+k}$ are
invariant (cf \cite{LL}). In addition, $p^*$ is essential to study the
separation property of solutions (see the following Remark).

\paragraph{Remark 1.3.} Let $u_{\mu}(r)$ be a regular solution of (\ref{1.6}).
Corollary 1.7 in \cite{YM} implies that, when $p \geq
\max\{p^*,p_{jl}\}$, $u_\mu(r)<u_s(r)$ for $r>0$, and
$u_{\mu_1}(r)<u_{\mu_2}(r)$ for $r>0$ as long as $\mu_1<\mu_2$.

The exponent $p^*$ also appears in the study of $\gamma$-Laplace
equations (cf. \cite{LLM} and \cite{YM}) and integral equations
involving Wolff potentials (cf. \cite{ChenLi}, \cite{MCL},
\cite{SL} and \cite{Vi}). In particular, it plays an important
role to investigate integrability, decay rates and intersection
properties of the positive entire solutions. In addition, this
exponent ensures that equation and energy $\|\cdot\|_{p+\gamma-1}$ are
invariant under the scaling transformation (cf \cite{LL}).

In particular, for the $\gamma$-Laplace equation
\begin{equation} \label{gamma}
-div(|\nabla u|^{\gamma-2}\nabla u)=K(x)u^p,
\quad  u>0 \quad in ~R^n,
\end{equation}
we write $p_{se}(\gamma)=\frac{n(\gamma-1)}{n-\gamma}$,
$p_{so}(\gamma)=\frac{n\gamma}{n-\gamma}-1$,
$p^*(\gamma)=\frac{n+\gamma}{n-\gamma}(\gamma-1)$,
$p_{jl}=\gamma-1+\gamma^2[n-\gamma-2-2\sqrt{(n-1)/(\gamma-1)}]^{-1}$
as $n>\frac{\gamma(\gamma+3)}{\gamma-1}$, and $p_{jl}=\infty$ as
$n \leq \frac{\gamma(\gamma+3)}{\gamma-1}$.

If $\gamma \in (1,2)$, $p_{se}(\gamma)< p^*(\gamma) < p_{so}(\gamma)$.
When $K(x) \equiv 1$, according to the Liouville theorem in \cite{SZ},
(\ref{gamma}) has no positive solution as $p<p_{so}(\gamma)$,
and $p^*(\gamma)$ does not make sense. When $K(x)$ is a double
bounded function, according to the result in \cite{LL},
(\ref{gamma}) has positive radial solutions as long as $p>p_{se}(\gamma)$.
Now, $p^*$ comes into play in studying integrability and decay rates
of positive solutions.


\vskip 5mm
Now, we state the results about the stable solutions.

\begin{theorem} \label{th1.4}
When $p<p_{jl}$, (\ref{1.1}) has no stable solution.
\end{theorem}

\begin{theorem} \label{th1.5}
When $p \geq p_{jl}$, the singular solution $u_{s}$ given by
(\ref{US}) is a stable solution of (\ref{1.1}).
\end{theorem}

\begin{theorem} \label{th1.6}
When $p=p_{so}$ or $p \geq \max\{p^*,p_{jl}\}$, all the
regular solutions of (\ref{1.6}) are stable solutions
of (\ref{1.1}) on $(R,\infty)$ for some $R>0$. When $p_{se}<p<p_{so}$,
the singular solutions introduced in Remark 1.1
are stable solution of (\ref{1.1}) on $(R,\infty)$ for some $R>0$.
\end{theorem}

\paragraph{Remark 1.4.} Theorem \ref{th1.5} shows that $u_s$ is also
a stable solution of (\ref{1.1}) on $(R,\infty)$ for some $R>0$ when $p \geq p_{jl}$.
Combining with Theorem \ref{th1.6}, we know that (\ref{1.1}) has stable solutions
on $(R,\infty)$ for some $R>0$ when $p \in (p_{se},p_{so}] \cup
[p_{jl},\infty)$. To our knowledge,
it is unknown whether (\ref{1.1}) has no stable solution on
$(R,\infty)$ for some $R>0$ when $p$ belongs to the gap
$(p_{so}, p_{jl})$.

\section{Regular solutions}

\begin{lem}\label{lem2.1}
Let $u$ be a regular solution of (\ref{1.6}). Then, $u' < 0$ for
$r>0$, and $u(r) \to 0$ as $r \to \infty$. Moreover, there are
positive constants $C_1,C_2$ such that for large $r$,
\begin{equation} \label{jie}
C_1r^{-\frac{n-2k}{k}} \leq u(r) \leq C_2r^{-\frac{2k}{p-k}}.
\end{equation}
\end{lem}

\begin{proof}
{\it Step 1.} Since $u$ is a positive solution of (\ref{1.1}),
$$
-\frac{1}{k}C_{n-1}^{k-1}(r^{n-k}|u'|^{k-1}u')'>0, \quad r>0.
$$
Integrating from $0$ to $R$ with $R>0$, we obtain
$$
R^{n-k}|u'(R)|^{k-1}u'(R)<0
$$
and hence $u'< 0$ is verified.

{\it Step 2.} By $u>0$ and $u'< 0$ for $r>0$, we know that $\lim\limits_{r
\to \infty}u(r)$ exists and hence is nonnegative. Suppose
that $\lim\limits_{r\rightarrow\infty}u(r) > 0$,
then there exists a constant $c>0$ such that $u \geq c$, and hence
$$
-\frac{1}{k}C_{n-1}^{k-1}(r^{n-k}|u'|^{k-1}u')' \geq c^pr^{n-1}.
$$
Integrating from $0$ to $R$, we obtain
$$
R^{n-k}|u'(R)|^{k-1}u'(R) \leq -CR^n.
$$
Here $C>0$ is independent of $R$. This result, together with $u'< 0$,
implies $u'(R) \leq -CR$. Integrating again yields
$$
u(r) \leq u(0)-Cr^2.
$$
Letting $r \to \infty$, we see a contradiction with $u>0$. This shows
that $u(r) \to 0$ as $r \to \infty$.

{\it Step 3.} According to the results in \cite{La} or \cite{PV2},
the regular solution of (\ref{1.6}) satisfies
\begin{equation}\label{Wf}
c_1W_{\frac{2k}{k+1},k+1}(u^p)(x) \leq u(|x|) \leq c_2[\inf_{x \in
R^n} u(|x|) +W_{\frac{2k}{k+1},k+1}(u^p)(x)],
\end{equation}
where $c_1,c_2$ are positive constant, and
$W_{\frac{2k}{k+1},k+1}(u^p)$ is the Wolff potential of $u^p$.
Namely,
$$
W_{\frac{2k}{k+1},k+1}(u^p)(x) =\int_0^\infty
(\frac{\int_{B_t(x)}u^p(|y|)dy}{t^{n-2k}})^{\frac{1}{k}}\frac{dt}{t}.
$$
Therefore, for large $|x|$,
$$
u(|x|) \geq c\int_{|x|+1}^\infty
(\frac{\int_{B_1(0)}u^p(|y|)dy}{t^{n-2k}})^{\frac{1}{k}}\frac{dt}{t}
\geq c\displaystyle\int_{|x|+1}^\infty
t^{\frac{2k-n}{k}}\frac{dt}{t} =c|x|^{\frac{2k-n}{k}}.
$$
Since $u$ is radially symmetric and decreasing, we can also get
$$
u(|x|) \geq c\int_0^{|x|/2} (\frac{\int_{B_{|x|}(0)\cap
B_t(x)}u^p(|y|)dy}{t^{n-2k}})^{\frac{1}{k}}\frac{dt}{t} \geq
cu^{\frac{p}{k}}(|x|)|x|^2,
$$
which implies that $u(|x|) \leq c|x|^{\frac{2k}{k-p}}$ for large
$|x|$.

The proof is complete.
\end{proof}

\subsection{Proof of Theorem \ref{th1.1}}

Let $p<p_{so}$. Assume that (\ref{1.6}) has a positive regular solution
$u$, we will deduce a contradiction.

{\it Step 1.} By (\ref{jie}), there exists $R>0$ such that
$u(r)\leq Cr^{-\frac{2k}{p-k}}$ for $r>R$. Thus,
\begin{equation}\label{3.1}
\int^{\infty}_{0}r^{n-1}u^{p+1}dr
\leq C(R)+\displaystyle\int^{\infty}_{R}r^{n-\frac{2(p+1)k}
{p-k}}\frac{dr}{r}<\infty.
\end{equation}

{\it Step 2.} Let $\varphi \in C^\infty(0,\infty)$ satisfy $\varphi(r)=1$ when
$r \in (0,1]$, $\varphi(r)=0$ when $r \in [2,\infty)$, and $0 \leq \varphi \leq 1$.
Write $\varphi_R(r)=\varphi(\frac{r}{R})$. Multiply (\ref{1.1}) by
$u\varphi_R^{k+1}$ and integrate on $(0,\infty)$. By the initial value condition
in (\ref{1.6}), we get
\begin{equation}\label{li5}
\begin{array}{ll}
\displaystyle\int_0^\infty r^{n-k} |u'|^{k+1}\varphi_R^{k+1}dr
&=\displaystyle\frac{k}{C_{n-1}^{k-1}}\int_0^\infty r^{n-1}u^{p+1}\varphi_R^{k+1}dr\\[3mm]
&-(k+1)\displaystyle\int_0^\infty r^{n-k}u\varphi_R^k |u'|^{k-1} u'\varphi'_R dr.
\end{array}
\end{equation}
By the Young inequality and the H\"older inequality, for a small $\epsilon>0$,
there holds that
\begin{equation}\label{li7}
\begin{array}{ll}
&|\displaystyle\int_0^\infty r^{n-k}u\varphi_R^k |u'|^{k-1}
u'\varphi'_R dr| \leq \epsilon \displaystyle\int_0^\infty r^{n-k}
|u'|^{k+1}\varphi_R^{k+1}dr\\[3mm]
& \quad \quad +\displaystyle\frac{C_\epsilon}{R^{k+1}}
(\int_0^\infty r^{n-1}u^{p+1}dr)^{\frac{k+1}{p+1}} (\int_R^{2R}
r^{\theta+1}\frac{dr}{r})^{\frac{p-k}{p+1}},
\end{array}
\end{equation}
where $\frac{p-k}{p+1}\theta=n-k-(n-1)\frac{k+1}{p+1}$. Therefore,
$(\theta+1)\frac{p-k}{p+1}-(k+1)=n\frac{p-k}{p+1}-2k<0$ by virtue of $p<p_{so}$.
Letting $R \to \infty$, we deduce from (\ref{3.1}), (\ref{li5}) and (\ref{li7})
that
\begin{equation}\label{li6}
\int_0^\infty r^{n-k}|u'|^{k+1}dr<\infty.
\end{equation}

{\it Step 3.}
Multiplying (\ref{1.1}) by $u$ and integrating on $(0,R)$, we obtain that
\begin{equation}\label{hjia}
\int_0^R r^{n-k}|u'|^{k+1}dr-R^{n-k}u(R)|u'(R)|^{k-1}u'(R)
=\frac{k}{C_{n-1}^{k-1}}\int_0^R r^{n-1}u^{p+1}dr.
\end{equation}
By (\ref{li6}) and (\ref{3.1}), there exists $R_j \to \infty$ such that
\begin{equation}\label{huij}
R_j^{n-k+1}|u'(R_j)|^{k+1}+R_j^n u^{p+1}(R_j) \to 0.
\end{equation}
Therefore, by $p<p_{so}$,
$$
R_j^{n-k}u(R_j)|u'(R_j)|^{k-1}u'(R_j) \to 0, \quad as ~R_j \to \infty.
$$
Inserting this result into (\ref{hjia}) and letting $R=R_j \to \infty$, we obtain
\begin{equation}\label{li8}
\int_0^\infty r^{n-k} |u'|^{k+1}dr
=\frac{k}{C_{n-1}^{k-1}}\int_0^\infty r^{n-1}u^{p+1}dr
\end{equation}

{\it Step 4.}
Multiplying (\ref{1.1}) by $ru'$ and integrating on $(0,R)$, we have
the Pohozaev type equality
\begin{equation}\label{li9}
\begin{array}{ll}
&-\displaystyle\frac{n-2k}{k+1}\int_0^R r^{n-k}|u'|^{k+1}dr
+\frac{k}{C_{n-1}^{k-1}}\frac{n}{p+1} \int_0^R r^{n-1}u^{p+1}dr\\[3mm]
&=\displaystyle\frac{k}{k+1}R^{n-k+1}|u'(R)|^{k+1}+\frac{k}{(p+1)C_{n-1}^{k-1}}
R^nu^{p+1}(R).
\end{array}
\end{equation}
By (\ref{huij}), the right hand side of (\ref{li9}) converges to
zero when $R=R_j \to \infty$. Letting $R=R_j \to \infty$ in (\ref{li9}) and
using (\ref{li8}), we can see $\frac{n-2k}{k+1}
=\frac{n}{p+1}$, which contradicts with $p<p_{so}$.

\subsection{Proof of Theorem \ref{th1.3}}

When $k=1$, the proof of the slow decay is based on the comparison principle
(cf. Lemma 2.20 and Theorem 2.25 in \cite{LN}).
For the quasilinear equation (\ref{1.1}), we use the monotony inequality
replacing the comparison principle.

\begin{lem} \label{lem2.*1}
Let $u(r)$ be a regular solution of (\ref{1.6}). If
$u(r)=O(r^{-\frac{2k}{p-k}-\varepsilon})$ with some $\varepsilon
\in (0,\frac{n-2k}{k}-\frac{2k}{p-k})$ for large $r$, then
$u(r)=O(r^{(2k-n)/k})$ for large $r$.
\end{lem}

\begin{proof}
If $u(r)=O(r^{-\frac{2k}{p-k}-\varepsilon})$ for large $r$, we can find
a large $R>0$ such that as $r>R$,
\begin{equation}\label{li4}
u(r) \leq Cr^{-\frac{2k}{p-k}-\varepsilon}.
\end{equation}

By Lemma \ref{lem2.1}, $\inf_{[0,\infty)}u(r)=0$. Using (\ref{Wf}) we have
$$
u(|x|) \leq C(I_1+I_2+I_3),
$$
where
$$
I_1=\int_0^{\frac{|x|}{2}}
(\frac{\int_{B_t(x)}u^p(|y|)dy}{t^{n-2k}})^{\frac{1}{k}}\frac{dt}{t},
$$
$$
I_2=\int_{\frac{|x|}{2}}^\infty
(\frac{\int_{B_t(x)\cap B_R(0)}u^p(|y|)dy}{t^{n-2k}})^{\frac{1}{k}}\frac{dt}{t},
$$
$$
I_3=\int_{\frac{|x|}{2}}^\infty
(\frac{\int_{B_t(x)\setminus B_R(0)}u^p(|y|)dy}{t^{n-2k}})^{\frac{1}{k}}\frac{dt}{t}.
$$
For sufficiently large $|x|$, we can deduce from (\ref{li4}) that
$$\begin{array}{ll}
&I_1 \leq C|x|^{-\frac{p}{k}(\frac{2k}{p-k}+\varepsilon)} \displaystyle
\int_0^{\frac{|x|}{2}} (\frac{\int_{B_t(x)}dy}{t^{n-2k}})^{\frac{1}{k}} \frac{dt}{t}
\leq C|x|^{-\frac{2k}{p-k}-\varepsilon\frac{p}{k}},\\[3mm]
&I_2 \leq C(|B_R(0)|u^p(0))^{1/k}\displaystyle\int_{\frac{|x|}{2}}^\infty t^{\frac{2k-n}{k}}\frac{dt}{t}
\leq C|x|^{-\frac{n-2k}{k}},\\[3mm]
&I_3 \leq \displaystyle\int_{\frac{|x|}{2}}^\infty
(\frac{\int_{B_{t+|x|}(0)\setminus B_R(0)}u^p(|y|)dy}{t^{n-2k}})^{\frac{1}{k}}\frac{dt}{t}
\leq C|x|^{-\frac{2k}{p-k}-\varepsilon\frac{p}{k}}.
\end{array}
$$
These estimates show that
$u(r) \leq C(r^{-\frac{n-2k}{k}}+r^{-\frac{2k}{p-k}-\varepsilon\frac{p}{k}})
\leq Cr^{-\frac{2k}{p-k}-\varepsilon\frac{p}{k}}$.
Replacing (\ref{li4}) by this result to estimate $I_1,I_2$ and $I_3$ as we
have done above, we get
$$
u(r) \leq C(r^{-\frac{n-2k}{k}}+r^{-\frac{2k}{p-k}-\varepsilon(\frac{p}{k})^2})
\leq Cr^{-\frac{2k}{p-k}-\varepsilon(\frac{p}{k})^2}.
$$
By iterating $m$ times, we can obtain
$$
u(r) \leq C(r^{-\frac{n-2k}{k}}+r^{-\frac{2k}{p-k}-\varepsilon(\frac{p}{k})^m}).
$$
Clearly, there exists a sufficiently large $m_0$ such that $\frac{n-2k}{k} \leq
\frac{2k}{p-k}+\varepsilon(\frac{p}{k})^{m_0}$. Thus, after $m_0$ steps,
we derive that,
$$
u(r) \leq Cr^{-\frac{n-2k}{k}} \quad for~ large~ r.
$$
Lemma \ref{lem2.*1} is proved.
\end{proof}

\begin{lem} \label{lem2.*3}
Let $u(r)$ be a regular solution of (\ref{1.6}). If
$u(r)=o(r^{-\frac{2k}{p-k}})$
for large $r$, then $u(r)=O(r^{(2k-n)/k})$ for large $r$.
\end{lem}

\begin{proof}
{\it Step 1.}
Let $\varphi(r) \in C^1(0,\infty)$ satisfy $\lim_{r \to \infty}
r^{n-\frac{2pk}{p-k}}\varphi(r)=0$.
Integrating (\ref{1.1}) from $0$ to $r$, we have
\begin{equation}\label{jiajia}
|u'|^{k-1}u'=-\frac{k}{C_{n-1}^{k-1}}r^{k-n}\displaystyle\int^r_{0}s^{n-1}u^p(s)ds.
\end{equation}
Thus, by $u(r)=o(r^{-\frac{2k}{p-k}})$
when $r \to \infty$, it follows that
\begin{equation}\label{huihui}
r^{n-k}|u'(r)|^k \varphi(r) \to 0.
\end{equation}

Multiply (\ref{1.1}) by $\varphi$ and integrate from $R$ to $\infty$.
By (\ref{huihui}), we obtain that
\begin{equation} \label{xiujuan}
\begin{array}{ll}
&\displaystyle\int_R^\infty r^{n-k} |u'|^{k-1}(u') \varphi' dr\\[3mm]
&=-R^{n-k}|u'(R)|^{k-1}u'(R) \varphi(R)
+\displaystyle\frac{k}{C_{n-1}^{k-1}}\int_R^\infty r^{n-1}u^p \varphi dr.
\end{array}
\end{equation}

Write $h(r):=c_*r^{-\theta}$, where $c_*$ is a positive constant determined later,
and $\theta:=\frac{2k}{p-k}+\epsilon_0$ with suitably small
$\epsilon_0>0$. By simply calculating and integrating by parts, we
get
$$\begin{array}{ll}
&\displaystyle\int_R^\infty r^{n-k}|h'|^{k-1}h'\varphi' dr
=-(c_*\theta)^k \int_R^\infty r^{n-k(\theta+2)}\varphi' dr\\[3mm]
&=(c_*\theta)^k[n-k(\theta+2)]
\displaystyle\int_R^\infty r^{n-1-k(\theta+2)}\varphi dr
+(c_*\theta)^k R^{n-k(\theta+2)} \varphi(R).
\end{array}
$$
Subtracting this result from (\ref{xiujuan}) yields
\begin{equation}\label{xiuli}
\begin{array}{ll}
&\displaystyle\int_R^\infty r^{n-k}[|u'|^{k-1}u'-|h'|^{k-1}h']\varphi' dr\\[5mm]
&=[(c_*\theta)^k R^{n-k(\theta+2)}+ R^{n-k}|u'(R)|^{k-1}u'(R)] \varphi(R)\\[3mm]
& +\displaystyle\int_R^\infty r^{n-1}[\frac{k u^p}{C_{n-1}^{k-1}}
-\frac{(c_*\theta)^k[n-k(\theta+2)]}{r^{k(\theta+2)}}]\varphi dr.
\end{array}
\end{equation}

{\it Step 2.}
In view of $k(\theta+2)=\frac{2pk}{p-k}+k\epsilon_0$, we can find
$\eta_0 \in (0,k\epsilon_0/p)$ such that
\begin{equation}\label{li0}
k(\theta+2)>\frac{2pk}{p-k}+p\eta_0.
\end{equation}

Since $u \in C^2$ is decreasing and $u(r)=o(r^{-\frac{2k}{p-k}})$
for large $r$, then either there exist positive constants $c_1,c_2$ such that
\begin{equation} \label{li1}
c_1r^{-\frac{2k}{p-k}} \geq u(r) \geq c_2r^{-\frac{2k}{p-k}-\eta_0}
\end{equation}
when $r$ is suitably large,
or $\lim_{r \to \infty}u(r)r^{\frac{2k}{p-k}+\eta_0}=0$, which implies that
there exists $\eta \in (0,\eta_0)$ such that for large $r$,
\begin{equation}\label{li2}
u(r) \leq cr^{-\frac{2k}{p-k}-(\eta_0-\eta)}.
\end{equation}

If (\ref{li2}) is true, Lemma \ref{lem2.*3}
can be proved easily by Lemma \ref{lem2.*1}.

In the following, we assume that (\ref{li1}) is true.
Take $\varphi=r^{-m}(u-h)_+$ in (\ref{xiuli}), where
$m>n-\frac{2pk}{p-k}$ is sufficiently large. Then,
\begin{equation}\label{jia1}
\begin{array}{ll}
&\displaystyle\int_R^\infty r^{n-k-m}[|u'|^{k-1}u'-|h'|^{k-1}h'][(u-h)_+]' dr\\[5mm]
&=[(c_*\theta)^k R^{n-k(\theta+2)-m}+ R^{n-k-m}|u'(R)|^{k-1}u'(R)] [u(R)-h(R)]_+\\[3mm]
& +\displaystyle\int_R^\infty r^{n-m-1}[\frac{k u^p}{C_{n-1}^{k-1}}
-\frac{(c_*\theta)^k[n-k(\theta+2)]}{r^{k(\theta+2)}}](u-h)_+ dr\\[3mm]
&+m\displaystyle\int_R^\infty r^{n-k-m-1}[|u'|^{k-1}u'-|h'|^{k-1}h'](u-h)_+ dr.
\end{array}
\end{equation}

By (\ref{jiajia}), (\ref{li1}) and (\ref{li0}), for any
$\delta \in (0,1)$, we can find $R_0>0$ such that as $r \geq R_0$,
$|h'|^k \leq \delta |u'|^k$. Therefore, the last term of the right
hand side of (\ref{jia1}) with $R=R_0$ is not larger than
$m(1-\delta)\int_{R_0}^\infty r^{n-k-m-1}|u'|^{k-1}u'(u-h)_+ dr$.
Choose $c_*=u(R_0)R_0^{\theta}$ to ensure $u(R_0)=h(R_0)$.
Therefore, the first term of the right
hand side of (\ref{jia1}) with $R=R_0$ is zero.
Thus, from (\ref{jia1}) with $R=R_0$ it follows that
\begin{equation}\label{jia2}
\begin{array}{ll}
&\displaystyle\int_{R_0}^\infty r^{n-k-m}[|u'|^{k-1}u'-|h'|^{k-1}h'][(u-h)_+]' dr\\[5mm]
&\leq \displaystyle\int_{R_0}^\infty r^{n-m-1}[\frac{k u^p}{C_{n-1}^{k-1}}
-m(1-\delta)r^{-k}|u'|^k](u-h)_+ dr.
\end{array}
\end{equation}
By (\ref{jiajia}) and the monotonicity of $u(r)$, there holds
$$
r^{-k}|u'(r)|^k \geq \frac{k}{C_{n-1}^{k-1}}r^{-n}u^p(r)
\int_0^r s^{n-1}ds \geq \frac{ku^p(r)}{nC_{n-1}^{k-1}}.
$$
Taking $m$ suitably large, we obtain that
the right hand side of (\ref{jia2}) is not larger than zero.
In view of the monotony inequality
$(|a|^{k-1}a-|b|^{k-1}b)(a-b) \geq 2^{k-1}|a-b|^{k+1}$, we obtain
from (\ref{jia2}) that
$$
\int_{R_0}^\infty r^{n-k}([(u-h)_+]')^{k+1} dr \leq 0,
$$
which implies $[u(r)-h(r)]_+ \equiv Constant$ for $r \geq R_0$. In view of $u(R_0)=h(R_0)$,
it follows $Constant=0$, which implies $u(r) \leq h(r)$ for $r \geq R_0$.
Applying Lemma \ref{lem2.*1}, we can also see the conclusion of Lemma \ref{lem2.*3}.
\end{proof}

\textbf{Proof of Theorem \ref{th1.3}.}
Let $p>p_{so}$.

{\it Step 1.} By Lemma \ref{lem2.1}, we see that $u(r) \leq Cr^{-\frac{2k}{p-k}}$ for
large $r$. We claim that there exists $c>0$ such that $u(r) \geq cr^{-\frac{2k}{p-k}}$ for
large $r$.

Otherwise, $\lim_{r \to \infty}u(r)r^{\frac{2k}{p-k}}=0$. By Lemma
\ref{lem2.*3} it follows that $u(r)=O(r^{\frac{2k-n}{k}})$ for
large $r$. Thus, $V(x) \in L^{p+1}(R^n) \cap C^2(R^n)$ (here
$V(x)=-u(|x|)$). According to Theorem 4.4 in \cite{Lei2}, we know
$p=p_{so}$, which contradicts with $p>p_{so}$.

{\it Step 2.}
We define by scaling a new function
$$
w(r)=\mu^{\frac{2k}{p-k}}u(\mu r),\quad \mu>0.
$$
By a direct calculation, we see that $w$ still satisfies (\ref{1.1}).
Applying the initial value conditions, we can obtain the second conclusion
of Theorem \ref{th1.3}.

\paragraph{Remark 2.1.}
Let $u_{\mu}(r)$ be a regular solution of (\ref{1.6}) with $p>p_{so}$.
When $p \geq p^*$, Miyamoto used the technique of
phase plane analysis to show that
$u_{\mu}(r)/u_s(r) \to 1$ as $r \to \infty$ (cf. Lemma 2.5 in \cite{YM}).
When $p \geq p_{so}$, Theorem \ref{th1.3} shows that the decay rate of $u_\mu$
is the same as that of $u_s$.
Furthermore, if $\lim_{r \to \infty}u(r)r^{-\frac{2k}{p-k}}$ exists, then
it must be $A$ which is introduced in (\ref{US}). In fact, integrating
(\ref{1.1}) twice yields
$$
u(r)=u(0)-(\frac{k}{C_{n-1}^{k-1}})^{1/k}\int_0^r
[t^{k-n}\int_0^t s^{n-1}u^p(s)ds]^{1/k}dt.
$$
Write $B:=\lim_{r \to \infty}\frac{u(r)}{r^{\frac{2k}{p-k}}}$.
Using the L'Hospital principle twice, we get
$$
B^k=\frac{k}{C_{n-1}^{k-1}}(\frac{2k}{p-k})^{-k}
\frac{\int_0^r s^{n-1}u^p(s)ds}{r^{n-\frac{2pk}{p-k}}}
=\frac{k}{C_{n-1}^{k-1}}(\frac{2k}{p-k})^{-k}
(n-\frac{2pk}{p-k})^{-1}B^p,
$$
which implies $B=A$.

\section{Stable solutions}

\subsection{Proof of Theorem \ref{th1.4}}

{\it Step 1.} We claim that for every
$\gamma\in [1,\frac{2p+2\sqrt{p(p-k)}-k}{k})$ and any integer
$m\geq  max\{{\frac{p+\gamma}{p-k},2}\}$, there exists a constant
$C>0$ such that for any $\psi \in W_*$, there holds
\begin{equation}\label{2.7}
\displaystyle\int^{R}_{0} r^{n-1}u^{p+\gamma}\psi^{m(k+1)}dr\leq
C\int^{R}_{0}(r^{\frac{(n-k)(p+\gamma)-(n-1)(\gamma+k)}{p+\gamma}}
|\psi'|^{k+1})^{\frac{p+\gamma}{p-k}}dr.
\end{equation}

{\it Proof of (\ref{2.7}).}
Let $\psi \in W_*$ be a cut-off function such that
$0 \leq \psi \leq 1$ and
$$
 \psi(r)=
\left\{
 \begin{array}{ll}
 1, \quad if ~r\leq R/2,\\
 0, \quad if ~r\geq R .
 \end{array}
 \right.
 $$
Clearly, there exists a constant $C>0$ such that $|\psi'|\leq \frac{C}{R}$.

Taking $\varphi=u^\gamma \psi^{m(k+1)}$ in (\ref{weak1}), we get
$$\begin{array}{ll}
& \displaystyle\frac{\gamma}{k}C_{n-1}^{k-1}\int^{R}_{0}
r^{n-k}|u'|^{k+1}u^{\gamma-1}\psi^{m(k+1)}dr\\[3mm]
&\quad\leq \displaystyle\frac{k+1}{k}C_{n-1}^{k-1}\int^{R}_{0}
r^{n-k}|u'|^{k}u^\gamma\psi^{mk}|(\psi^m)'|dr +\int^{R}_{0}
r^{n-1}u^{p+\gamma}\psi^{m(k+1)}dr.
\end{array}
$$
Using the Young inequality to the first term of the right hand side,
we can obtain that for any small $\varepsilon>0$,
\begin{equation}\label{xiu2}
\begin{array}{ll}
&(\displaystyle\frac{\gamma}{k}C_{n-1}^{k-1}-\varepsilon^2)\displaystyle\int^{R}_{0}
r^{n-k}u^{\gamma-1}|u'|^{k+1}\psi^{m(k+1)}dr\\[3mm]
&\quad\leq \displaystyle C_{\varepsilon}\int^{R}_{0}
r^{n-k}u^{\gamma+k}|(\psi^m)'|^{k+1}dr+\int^{R}_{0}
r^{n-1}u^{p+\gamma}\psi^{m(k+1)}dr.
\end{array}
\end{equation}

Taking $\varphi=u^{\frac{\gamma+1}{2}}\psi^{\frac{m(k+1)}{2}}$
in (\ref{weak2}), we have
\begin{equation}\label{2.4}
\begin{array}{ll}
&\displaystyle p\int^{R}_{0} r^{n-1}u^{p+\gamma}\psi^{m(k+1)}dr\\[3mm]
&\leq \displaystyle\frac{C_{n-1}^{k-1}(\gamma+1)^2}{4}\int^{R}_{0}
 r^{n-k}u^{\gamma-1}|u'|^{k+1}\psi^{m(k+1)}dr\\[5mm]
&\displaystyle+\frac{C_{n-1}^{k-1}(k+1)^2}{4}\int^{R}_{0}r^{n-k}
|u'|^{k-1}u^{\gamma+1}\psi^{m(k-1)}|(\psi^m)'|^2dr\\[5mm]
&+\displaystyle\frac{C_{n-1}^{k-1}(\gamma+1)(k+1)}{2}\int^{R}_{0}
r^{n-k}|u'|^ku^\gamma\psi^{mk}|(\psi^m)'|dr.
\end{array}
\end{equation}
Using the Young inequality to the second and the third terms
of the right hand side of (\ref{2.4}), we get
\begin{equation}\label{2.5}
\begin{array}{ll}
&\quad\displaystyle p\int^{R}_{0}
r^{n-1}u^{p+\gamma}\psi^{m(k+1)}dr\\[3mm]
&\leq(\displaystyle\frac{C_{n-1}^{k-1}(\gamma+1)^2}
{4}+\varepsilon^2)\int^{R}_{0}r^{n-k}|u'|^{k+1}u^{\gamma-1}
\psi^{m(k+1)}dr\\[3mm]
&+C_{\varepsilon}\displaystyle\int^{R}_{0}
r^{n-k}|(\psi^m)'|^{k+1}u^{\gamma+k}dr.
\end{array}
\end{equation}
Combining (\ref{xiu2}) and (\ref{2.5}), we obtain by the H\"older
inequality that
\begin{equation}\label{xiu3}
\begin{array}{ll}
&\displaystyle[p-(\frac{C_{n-1}^{k-1}(\gamma+1)^2}{4}+\varepsilon^2)
\frac{1}{\frac{\gamma}{k}C_{n-1}^{k-1}-\varepsilon^2}]\int^{R}_{0}
r^{n-1}u^{p+\gamma}\psi^{m(k+1)}dr\\[3mm]
&\leq C \displaystyle\int^{R}_{0}
r^{n-k}|(\psi^m)'|^{k+1}u^{\gamma+k}dr\\[3mm]
&\leq \displaystyle C[\int^{R}_{0}(r^{(n-1)\frac{\gamma+k}{p+\gamma}}
u^{\gamma+k}\psi^{(m-1)(k+1)})
^{\frac{p+\gamma}{\gamma+k}}dr]
^{\frac{\gamma+k}{p+\gamma}}\\[3mm]
&\quad\cdot
[\displaystyle\int^{R}_{0}(r^{\frac{(n-k)(p+\gamma)-(n-1)
(\gamma+k)}{p+\gamma}}|\psi'|^{k+1})^
{\frac{p+\gamma}{p-k}}dr]^{\frac{p-k}{p+\gamma}}.
\end{array}
\end{equation}
In view of $\gamma\in [1,\frac{2p+2\sqrt{p(p-k)}-k}{k})$,
$\lim\limits_{\varepsilon \to 0}
[p-(\frac{C_{n-1}^{k-1}(\gamma+1)^2}{4}+\varepsilon^2)\frac{1}
{\frac{\gamma}{k}C_{n-1}^{k-1}-\varepsilon^2}]
=p-\frac{k(\gamma+1)^2}{4\gamma}>0$.
Therefore, the coefficient of the left hand side of (\ref{xiu3})
is positive as long as $\varepsilon$ is sufficiently small.
Therefore, noting $(m-1)(k+1)\frac{p+\gamma}{\gamma+k}\geq m(k+1)$ which is implied
by $m\geq max\{{\frac{p+\gamma}{p-k},2}\}$,
we can deduce (\ref{2.7}) from (\ref{xiu3}) by the Young inequality.

{\it Step 2.}
By the definition of $\psi$, from (\ref{2.7}) we can deduce that
\begin{equation}\label{2.10}
\int^{R}_{0}r^{n-1}u^{p+\gamma}\psi^{m(k+1)}dr\\[3mm]
\leq CR^{n+1-\frac{(2k+1)(p+\gamma)-(\gamma+k)}{p-k}}.
\end{equation}
When $n+1-\frac{(2k+1)(p+\gamma)-(\gamma+k)}{p-k}<0$, the desired claim follows by
letting $R\rightarrow \infty$.

Consider a real-valued function
$$
f(t)=\frac{(2k+1)(t+\gamma(t))-(\gamma(t)+k)}{t-k},\quad t \in (k,\infty),
$$
where $\gamma(t)=\frac{2t+2\sqrt{t(t-k)}-k}{k}$.
Clearly, we know $f(t)$ is a strictly decreasing function (by virtue of
$f'(t)<0$ on $(k,\infty)$),
satisfying $\lim\limits_{t\rightarrow k}f(t)=\infty$ and $\lim\limits_{t\rightarrow
\infty}f(t)=2k+9$. Therefore, we consider separately two cases:
$n\leq 2k+8$, and $n\geq 2k+9$.

Case I: $n\leq 2k+8$. In view of $p>p_{se}$, there exists
$\gamma\in [1,\frac{2p+2\sqrt{p(p-k)}-k}{k})$ such that $n+1-\frac{(2k+1)
(p+\gamma)-(\gamma+k)}{p-k}<0$ is true.

Case II: $n\geq 2k+9$. In view of $p>p_{se}$, there exists a unique $p_{0}>k$
such that $n+1=f(p_{0})$ since $f(t)$ is decreasing in $(k,\infty)$.
Therefore, $p_{0}$ satisfies
\begin{equation}\label{2.11}
(n-2k)(n-2k-8)p^2_{0}-2k[n^2-2(k+3)n+4k]p_{0}+k^2(n-2)^2=0,
\end{equation}
and
\begin{equation}\label{2.12}
(n-2k-4)p_{0}-(n-2)k>4(p_{0}-k).
\end{equation}
The roots of equation (\ref{2.11}) are
\begin{equation}\label{2.13}
p_{1}=\frac{k[n^2-2(k+3)n+4k]+4k\sqrt{2(k+1)n-4k}}{(n-2k)(n-2k-8)},
\end{equation}
\begin{equation}\label{2.14}
p_{2}=\frac{k[n^2-2(k+3)n+4k]-4k\sqrt{2(k+1)n-4k}}{(n-2k)(n-2k-8)}.
\end{equation}
Inequality (\ref{2.12}) implies $p_{0}>p_{2}$, and hence we take
$p_{0}=p_{1}$ (it equals exactly $p_{jl}$). Thus, when $p<p_{jl}$,
there exists $\gamma\in [1,\frac{2p+2\sqrt{p(p-k)}-k}{k})$
satisfying $n+1-\frac{(2k+1) (p+\gamma)-(\gamma+k)}{p-k}<0$.

No matter in Case I or Case II, letting $R\rightarrow\infty$ in (\ref{2.10}),
we can deduce $\int^{R}_{0}r^{n-1}u^{p+\gamma}dr \rightarrow 0$.
This contradiction shows that (\ref{1.1}) has no positive stable solution
as long as $p<p_{jl}$.

\subsection{Proof of Theorem \ref{th1.5}}

Let $u_{s}$ be the singular solution of (\ref{1.1}) given by (\ref{US}).
We will prove that the singular solution $u_{s}(r)$ is stable
when $n \geq 2k+9$ and $p\geq p_{jl}$.

First, we claim that $u_s$ satisfies (\ref{weak1}). In fact, by (\ref{serrin}),
the improper integral
$
\int_0^\infty r^{n-1}u_s^p \varphi dr \leq C\int_0^R r^{n-1-\frac{2pk}{p-k}}dr
<\infty.
$
Similarly, the left hand side of (\ref{weak1})
also makes sense. In addition, $u_s$ solves (\ref{1.1}). Multiply by the test function
$\varphi \in W_*$ and integrate from $0$ to $\infty$.
Noting $r^{n-k}|u'_s(r)|^k \to 0$ as $r \to 0$,
we know that the claim is true.

To prove that $u_{s}$ satisfies (\ref{weak2}), we observe firstly that
\begin{equation}\label{3.15}
\begin{array}{ll}
&\displaystyle p(\frac{2}{p-k})(n-\frac{2pk}{p-k})\leq
\frac{(n-2-\frac{2p(k-1)}{p-k})^2}{4}\\[5mm]
&\Leftrightarrow 8n(p^2-kp)-16kp^2\leq (n-2)^2(p^2-2kp+k^2)\\[3mm]
&\quad +4(k-1)^2p^2-4(k-1)(n-2)(p^2-kp)\\[3mm]
&\Leftrightarrow (n-2k)(n-2k-8)p^2-2k
(n^2-2(k+3)n+4k)p\\[3mm]
&\quad +k^2(n-2)^2 \geq 0\\[3mm]
&\Leftrightarrow p\in(-\infty, p_{2}]\bigcup[p_{jl},+\infty)
\end{array}
\end{equation}
where $p_{2}$ is defined in (\ref{2.14}). On the other hand, by
Definition 1.2, we have that for any $\phi \in C_c^\infty(R^n)$,
$$\begin{array}{ll}
&\displaystyle C_{n-1}^{k-1}\int_{R^{n}}\frac{1}{|x|^{k-1}}
|u_{s}'|^{k-1}|\nabla \phi|^2 dx -
p\displaystyle\int_{R^{n}}u_{s}^{p-1}\phi^2 dx\\[5mm]
&=\displaystyle C_{n-1}^{k-1}\int_{R^{n}}(\frac{1}{k}C_{n-1}^{k-1})
^{\frac{k-1}{p-k}}(\frac{2k}{p-k})^{\frac{(k-1)p}{p-k}}
(n-\frac{2pk}{p-k})^{\frac{k-1}{p-k}}
\frac{1}{|x|^{\frac{2p(k-1)}{p-k}}}|\nabla \varphi|^2 dx\\[5mm]
&\quad\quad\quad\quad -p\displaystyle\int_{R^{n}}(\frac{1}{k}C_{n-1}^{k-1})
^{\frac{p-1}{p-k}}(\frac{2k}{p-k})
^{\frac{(p-1)k}{p-k}}(n-\frac{2pk}{p-k})
^{\frac{p-1}{p-k}}\frac{1}{|x|^{\frac{2(p-1)k}{p-k}}}\varphi^2 dx\\[5mm]
&=C_{0}(\displaystyle\int_{R^{n}}\frac{1}{|x|^{\frac{2p(k-1)}{p-k}}}|\nabla\phi|^2 dx
-p(\frac{2}{p-k})(n-\frac{2pk}{p-k})\displaystyle\int_{R^{n}}\frac{1}
{|x|^{\frac{2(p-1)k}{p-k}}}\phi^2)dx,
\end{array}
$$
where
\begin{equation}\label{3.16}
C_{0}=\displaystyle C_{n-1}^{k-1}(\frac{1}{k}C_{n-1}^{k-1})
^{\frac{k-1}{p-k}}(\frac{2k}{p-k})^{\frac{(k-1)p}{p-k}}(n-\frac{2pk}{p-k})
^{\frac{k-1}{p-k}}.
\end{equation}
By $p\geq p_{jl}$, (\ref{3.15}) implies that
$$
\begin{array}{ll}
&\displaystyle\int_{R^{n}}\frac{1}{|x|^{\frac{2p(k-1)}{p-k}}}|\nabla\phi|^2dx
-p(\frac{2}{p-k})(n-\frac{2pk}{p-k})\displaystyle\int_{R^{n}}\frac{1}
{|x|^{\frac{2(p-1)k}{p-k}}}\phi^2dx\\[5mm]
&\geq \displaystyle\int_{R^{n}}\frac{1}{|x|^{\frac{2p(k-1)}{p-k}}}|\nabla\phi|^2dx
-\frac{(n-2-\frac{2p(k-1)}{p-k})^2}{4}\displaystyle\int_{R^{n}}\frac{1}
{|x|^{\frac{2(p-1)k}{p-k}}}\phi^2dx.
\end{array}
$$
It follows that
\begin{equation}\label{xuanyi}
Q_{u_{s}}(\varphi)> 0, \quad \forall \varphi\in W_*
\end{equation}
by the Caffarelli-Kohn-Nirenberg inequality (cf. \cite{CC})
\begin{equation}\label{CKN}
\int_{R^{n}}\frac{|\nabla \phi|^2}{|x|^{2a}} dx >
C_{a,b}\int_{R^{n}}\frac{\phi^2}{|x|^{2b}} dx,\quad
\forall \phi \in D^{1,2}_{a}(R^{n}),
\end{equation}
where $n\geq3$, $0\leq a<\frac{n-2}{2}$ and $a\leq b\leq a+1$, the
best constant $C_{a,b}$ is given by
$C_{a,b}=\displaystyle\frac{(n-2-2a)^2}{4}$. Here we take
$a=\frac{p(k-1)}{p-k}$ and $b=a+1$. This result shows that $u_{s}$
is a stable solution of (\ref{1.1}) when $n \geq 2k+9$ and $p\geq
p_{jl}$. The proof of Theorem \ref{th1.5} is complete.

\subsection{Proof of Theorem \ref{th1.6}}

{\it Step 1.}
When $p=p_{so}$,
all regular solutions $u_\rho$ of (\ref{1.6}) can be written as the form given
by (\ref{expl}). When $r$ is suitably large,
\begin{equation}\label{Ur}
u_\rho(r) \leq D_1 r^{-\frac{n-2k}{k}}, \quad |u'_\rho| \geq D_2 r^{-\frac{n-k}{k}},
\end{equation}
where $D_1,D_2$ are positive constants independent of $r$.
Thus,
$$
pu^{p-1}(r)=O(r^{-\frac{(k-1)n}{k}-4}),\quad as~ r \rightarrow
\infty.
$$
Therefore, we can find some $R>0$ such that for all $|x|>R$ and
$\phi \in C_c^\infty(R^n\setminus \overline{B_R(0)})$, there holds
$$
pu_{\rho}^{p-1}(|x|)\phi^2(x) < C^*|x|^{-\frac{(k-1)n}{k}-2}\phi^2(x),
$$
where $C^*=\frac{(n-2-\frac{k-1}{k}n)^2}{4}D_2^{k-1}C_{n-1}^{k-1}(\frac{n-2k}{k})^{k-1}$.
Thus,
\begin{equation}\label{Su}
\begin{array}{ll}
&\displaystyle C_{n-1}^{k-1}\int_{{R}^{n}}\frac{1}{|x|^{k-1}}
|u_{\rho}'|^{k-1}|\nabla \phi|^2dx -
p\displaystyle\int_{{R}^{n}}u_{\rho}^{p-1}\phi^2dx\\[5mm]
&\geq\displaystyle D_2^{k-1} C_{n-1}^{k-1}(\frac{n-2k}{k})^{k-1}\int_{{R}^{n}}
\frac{1}{|x|^{\frac{k-1}{k}n}}|\nabla \phi|^2dx\\[3mm]
&\quad  -C^*\displaystyle\int_{{R}^{n}}
\frac{1}{|x|^{\frac{k-1}{k}n+2}}\phi^2dx\\[3mm]
&=\displaystyle D_2^{k-1} C_{n-1}^{k-1}(\frac{n-2k}{k})^{k-1}(\displaystyle\int_{{R}^{n}}
\frac{1}{|x|^{\frac{k-1}{k}n}}|\nabla \phi|^2dx\\[3mm]
&\quad -\displaystyle\frac{(n-2-\frac{k-1}{k}n)^2}{4}
\displaystyle\int_{{R}^{n}}\frac{1}{|x|^{\frac{k-1}{k}n+2}}\phi^2)dx,
\end{array}
\end{equation}
and the right hand side is nonnegative by the
Caffarelli-Kohn-Nirenberg inequality (\ref{CKN}) with $a=\frac{k-1}{2k}n$
and $b=a+1$. Therefore, $Q_{u_{\rho}}(\varphi) \geq 0$ for every
$\varphi \in C_c^\infty(R,\infty)$. In addition, $u_\rho$ also satisfies
(\ref{weak1}). So the regular solution $u_{\rho}$ is stable on $(R, \infty)$.

{\it Step 2.}
Let $u_\mu$ (see (\ref{xiu0})) be a regular solution of (\ref{1.6})
with $p \geq \max\{p^*,p_{jl}\}$.
We claim that $u_\mu$ is stable on $(R,\infty)$ for some $R>0$.

We at first prove $\lim\limits_{r\rightarrow\infty}u'_{\mu}(r)/u'_{s}(r)=1$
when $p \geq \max\{p^*,p_{jl}\}$.

Clearly, $u'_{s}=-(\frac{1}{k}C_{n-1}^{k-1})^{\frac{1}{p-k}}(\frac{2k}{p-k})^{\frac{p}{p-k}}
(n-\frac{2pk}{p-k})^{\frac{1}{p-k}}r^{-\frac{p+k}{p-k}}$.

Combining with (\ref{jiajia}) and
using the L'Hospital principle, we get
$$
\begin{array}{ll}
\displaystyle\lim_{r\rightarrow\infty}(\frac{u'_{\mu}}{u'_{s}})^k&=\lim\limits_{r\rightarrow\infty}\frac{
\displaystyle\int^r_{0}s^{n-1}u_{\mu}^p(s)ds}
{(\frac{1}{k}C_{n-1}^{k-1})^{\frac{p}{p-k}}(\frac{2k}{p-k})^{\frac{pk}{p-k}}(n-\frac{2pk}{p-k})
^{\frac{k}{p-k}}r^{n-\frac{2pk}{p-k}}}\\[5mm]
&=\displaystyle\lim_{r\rightarrow\infty}\frac{r^{n-1}u_{\mu}^p(r)}{(\frac{1}{k}C_{n-1}^{k-1})^{\frac{p}{p-k}}
(\frac{2k}{p-k})^{\frac{pk}{p-k}}(n-\frac{2pk}{p-k})
^{\frac{p}{p-k}}r^{n-\frac{2pk}{p-k}-1}}\\[5mm]
&=\displaystyle\lim_{r\rightarrow\infty}\frac{u_{\mu}^p(r)}{u_s^p(r)}.
\end{array}
$$
By Remark 2.1, there holds
$\lim\limits_{r\rightarrow\infty}u_{\mu}'(r)/u'_{s}(r)=1$ when $p \geq
\max\{p^*,p_{jl}\}$. Thus, there exists sufficiently large $R>0$
such that as $r>R$,
$$
|u_{\mu}'(r)|^{k-1}=|u'_s(r)|^{k-1}+o(1)r^{-\frac{(k-1)(p+k)}{p-k}}.
$$
Therefore, by the strict inequality (\ref{xuanyi}), we can find a suitably small $\delta_0>0$
such that for any $\psi \in C_c^\infty(R^n \setminus \overline{B_R(0)})$,
$$\begin{array}{ll}
&C_{n-1}^{k-1}\displaystyle\int_{R^n}\frac{|u'_{\mu}(|x|)|^{k-1}}{|x|^{k-1}}|\nabla \phi|^2dx\\[3mm]
&=C_{n-1}^{k-1}\displaystyle\int_{R^n}\frac{|u'_s(|x|)|^{k-1}
+o(1)|x|^{-\frac{(k-1)(p+k)}{p-k}}}{|x|^{k-1}}|\nabla \phi|^2dx\\[3mm]
&\geq C_0\displaystyle[p(\frac{2}{p-k})(n-\frac{2pk}{p-k})+\delta_0+o(1)]
\displaystyle\int_{R^{n}}\frac{1}{|x|^{\frac{2(p-1)k}{p-k}}}\phi^2 dx\\[3mm]
&\geq p \displaystyle\int_{R^n}u_s(|x|)^{p-1}\phi^2dx.
\end{array}
$$
Here $C_0$ is the constant in (\ref{3.16}).
In view of $u_s(r)>u_{\mu}(r)$ for $r>R$ (see Remark 1.3), we can see $Q_{u_{\mu}}(\varphi)
\geq 0$ for any $\varphi \in C_c^\infty(R,\infty)$. In addition, $u_{\mu}$ satisfies
(\ref{weak1}). Thus, $u_{\mu}$ is stable on $(R,\infty)$ for some $R>0$.

{\it Step 3.} Let $U_s$ be a singular solution of (\ref{1.1}) with $p
\in (p_{se},p_{so})$ introduced in Remark 1.1. By the same as in the proof of
Theorem \ref{th1.5}, $U_s$ still satisfies (\ref{weak1}) since
$0$ is not the singular point in the improper integrals of (\ref{weak1})
which is implied by $\lim_{r \to 0}U_s(r)/u_s(r)=1$.

In addition, by an analogous
argument in Step 1, $U_s$ still satisfies (\ref{weak2}). In fact,
$\lim_{r \to \infty}U_s(r)r^{\frac{n-2k}{k}}=\lambda$ implies
\begin{equation}\label{jia5}
U_s(r) \leq Cr^{-\frac{n-2k}{k}} \quad for~large~r.
\end{equation}
On the other hand, by (\ref{jiajia}), the monotonicity of $U_s$,
and (\ref{jia5}), there holds
$$
|U_s'|^k \geq cr^{k-n}U_s^p(r)\int_0^r s^{n-1}ds
\geq cr^{k-p\frac{n-2k}{k}}
$$
for large $r$. Therefore, applying the Caffarelli-Kohn-Nirenberg inequality
(\ref{CKN}) with $a=p\frac{n-2k}{k}\frac{k-1}{2k}$ and $b=a+1$, we obtain by (\ref{jia5})
and $p>p_{se}$ that
$$\begin{array}{ll}
&\displaystyle\int_{R^n}\frac{|U_s'(|x|)|^{k-1}}{|x|^{k-1}}|\nabla \phi|^2dx
\geq c\int_{R^n}\frac{\phi^2 dx}{|x|^{p\frac{n-2k}{k}\frac{k-1}{k}+2}}\\[3mm]
&\geq c\displaystyle\int_{R^n}\frac{\phi^2 dx}{|x|^{(p-1)\frac{n-2k}{k}}}
\geq p\int_{R^n}U_s^{p-1}(|x|)\phi^2dx
\end{array}
$$
for any $\phi \in C_c^\infty(R^n \setminus \overline{B_R(0)})$
with suitably large $R$.


\end{document}